\documentclass[11pt]{paper}
\usepackage{tikz}
\usepackage{tikz}
\usepackage[centertags]{amsmath}
\usepackage{amsfonts}
\usepackage{amssymb}
\usepackage{amsthm}
\usepackage{graphics, graphicx}
\usepackage{newlfont}
\usepackage{tabularx}

\theoremstyle{plain}
\newtheorem{thm}{Theorem}[section]
\newtheorem{cor}[thm]{Corollary}
\newtheorem{lem}[thm]{Lemma}
\newtheorem{prop}[thm]{Proposition}
\theoremstyle{definition}

\theoremstyle{remark}

\theoremstyle{Example}

\theoremstyle{Conjecture}
\newtheorem{conj}{Conjecture}[section]

\theoremstyle{Question}
\newtheorem{quest}{Question}[section]

\numberwithin{equation}{section}

\begin{document}
\title{Sums of Polynomials and Clique Roots
}
\author{
Hossein Teimoori Faal
\\
Department of Mathematics and Computer Science,
\\
Allameh Tabatabai University, Tehran, Iran
\\
hossein.teimoori@atu.ac.ir
}

\date{17 December 2021}

\maketitle
%%% ----------------------------------------------------------------------

\begin{abstract}

In this paper, pursuing the same line 
of ideas in the proof of 
an old longstanding open conjecture
of \emph{Kadison-Singer}
, we introduce a key lemma which we call it the interlacing 
lemma which indicates a 
necessary condition for having 
a real root for sums of polynomials with (at least) one 
real root. 
Then, as an immediate application of this simple but potentially 
useful lemma we characterize 
several class of graphs which have only clique roots. 
Finally, we conclude our paper with several interesting 
open problems and 
conjectures for interested readers.

\end{abstract}

%%% ----------------------------------------------------------------------
\section{Introduction and Motivation}

The property of having only real roots for graphs polynomials of some special classes of graphs 
is of special interest for algebraic graph theorists in recent years. The reason is that 
having only real roots for a graph polynomial results 
in detecting many combinatorial features for 
the given graph. 
\\
One of the most interesting breaking through in mathematical sciences 
in recent years 
was the solution of the well-known 
\emph{Kadsion-Singer} conjecture which was 
open for more than $50$ years. 
The main idea behind the proof was 
introducing an interlacing family of polynomials 
(see \cite{HajiMehrabadi(1998)} and reference therein for more details). 
\\
We are not really aware of the origin of interlacing idea, but it seems that 
one of the early motivations comes from the famous \emph{Rolle}'s theorem in the  
calculus of one-variable functions. An immediate corollary of 
Rolle's theorem simply states that 
for a real polynomial $f(x)$, roots of 
$
f'(x)
$ 
interlace those of $f(x)$. 
\\ 
It seems that by using similar line of proofs for 
the Rolle's theorem, one can get the following simple but
interesting lemma. 

\begin{lem}\label{keylem1}
	Let $\{f_{i}(x) \}_{i}$ be a finite collection of polynomials with all real coefficients, a positive leading coefficient and at least one real roots. We also let $R_{i}$ and $r_{i}$ be the largest and second largest 
	real roots (respectively) of 
	$\{f_{i}(x) \}_{i}$. In the case $f_{i}(x)$
	has only one real root, by convention we put $r_{i}= - \infty$. If 
	$
	\max_{j} r_{j} \leq \min_{j} R_{j}
	$
	,
	then $f(x) = \sum_{j} f_{j}(x)$ has a real root 
	$R$ which satisfies 
	$
	\min_{j} R_{j} \leq R
	$ 	
	.
\end{lem}

The importance of the above \emph{key lemma} is that it simply gives interesting necessary conditions for having a real root for 
\emph{sums of polynomials}. 
Needless to say, finding necessary conditions for sums of polynomials to having a real roots can be interesting in several occasions where a new polynomial graph invariant can be expressed as a sum of polynomial invariant of its particular subgraphs. For instance, the first derivatives of several graph polynomials like characteristic polynomial, matching polynomial, independence polynomial and the clique polynomials of a given graph can be obtained as a sum of its vertex-deleted subgraphs \cite{StanBook1}.     
\\ 
In this paper, we will concentrate on the clique polynomial of a graph. By applying the key lemma to the combinatorial interpretations of the first and 
second derivatives of clique polynomials we obtain
several classes of graphs which have only clique roots. We finally conclude the paper with the couple of open questions and conjectures.

\section{Basic Definitions and Notations}

Throughout this paper, we will assume that our graphs are 
all finite, simple and undirected. For the definitions which are not appear here, one may refer to \cite{StanBook1}. 
\\ 
A subset of vertices of a graph $G$ that are pairwise adjacent is called a \emph{complete} 
subgraph or a \emph{clique} of $G$. A clique with $k$ vertices is called a $k$-\emph{clique}. The number of 
k-cliques will be denoted by $c_{k}(G)$. For a 
subset $S$ of vertices, the graph with vertex set 
$S$ and edges with end-vertices only on $S$ is called 
an \emph{induced subgraph} of $G$  
and denoted by $G[S]$. The set vertices adjacent to a vertex $v$ is called the (open) \emph{neighborhood} of 
$v$ and is denoted by $N(v)$. The subgraph obtained 
by deleting the vertex $v$ from $G$ will be denoted by $G-v$. 
In a similar way, a subgraph obtained by only removing an 
edge $e$ from $G$ is denoted by $G-e$. 
\\
A \emph{chord} of a cycle is an edge connecting two \emph{non-adjacent} vertices in the cycle. 
A \emph{chordal} graph is a graph that any cycle of length greater than \emph{three} has a chord. 
\\
For a given graph $G=(V,E)$, the \emph{ordinary generating} function of the number of cliques of $G$ is called the 
\emph{clique polynomial} of $G$ and is 
denoted by $C(G,x)$. More precisely, we have  
\begin{equation}
C(G,x) = \sum_{k=0}^{\omega(G)} c_{k}(G) x^{k},
\end{equation}
where 
$
\omega(G)
$
is the size of the largest clique of $G$. By convention, we may assume $c_{0}(G)=1$ for any 
graph $G$.  
The real root of the clique polynomial 
of $G$ is called the \emph{clique root}
of $G$. 
\\ 
The clique polynomial of a graph satisfies
the following vertex-recurrence relation:
\begin{equation}
C(G,x) = C(G-v,x) + C(G[N(v)],x). 
\end{equation} 
Similarly, we have the following edge-recurrence 
relation:
\begin{equation}
C(G,x) = C(G-e,x) + C(G[N(e)],x), 
\end{equation}
in which 
$
N(e)=N(u) \cap N(v)
$
, 
for $e=\{u,v\} \in E(G)$.

The following combinatorial interpretation of 
the first derivative of clique polynomial is 
given in \cite{Teimoori2016}.

\begin{equation}\label{DerivatCliq1}
\frac{d}{dx} C(G,x) = 
\sum_{v \in V(G)} C(G[N(v)],x). 
\end{equation}

In a similar way, one can obtain the following 
edge-version of the graph-theoretical interpretation 
of the second derivative of a clique polynomial. 

\begin{equation}\label{DerivatCliq2}
\frac{1}{2!}\frac{d^2}{dx^2} C(G,x) = 
\sum_{e \in V(G)} C(G[N(e)],x). 
\end{equation}

\section{Main Results} 

We first note that 
based on the recursive definition of chordal graphs using the idea of pasting two complete graphs along a clique, one can show that any $k$-connected chordal graphs has a clique root $-1$ of multiplicity $k$ \cite{Teimoori2016}.

%\begin{proof}[Sketch of proof]
%The proof is based on case analysis and the intermediate value theorem. 
%\end{proof}
Next, we need another key result. 
\begin{prop}\label{keylem2}
Let $T$ be a \emph{tree} on $n$ vertices. Then, the graph $T$ has only clique roots. Moreover, the largest and the second largest clique roots are $R=- \frac{1}{n-1}$ and 
$r=-1$, respectively. 	
\end{prop}
An immediate corollary of Lemma \ref{keylem1} and Proposition \ref{keylem2} is the following. 
\begin{cor}\label{keyForest1} 
Let $F$ be a \emph{forest} with non-trivial components. Then the greatest clique root $R_{F}$ of $F$ satisfies the inequality $R_{F} \geq -\frac{1}{n-1}$, where $n$ 
is the size the smallest component of $F$. 
\end{cor}

From now on, for simplicity of arguments, we will 
assume that our graphs are connected. 

\begin{prop}
Let $G$ be a triangle-free graph. Then $G$ has only 
clique roots. 	
\end{prop}

\begin{proof}
Since $G$ is triangle-free, the neighborhood of each vertex is an independent set. That is
$
C(G[N(v)],x)= 1 + r_{v}x
$	
where $r_{v}$ is the size of corresponding independent set of $v$. If $r$ denotes the smallest size of all those independent sets, then we clearly have 
\begin{equation}
-\infty = \max_{j} r_{j} < -1 \leq -\frac{1}{r} = 
\min_{j} R_{j}.
\end{equation}
Hence the conditions of Lemma \ref{keylem1} are true and the polynomial 
$$
\sum_{v \in V(G)} C(G[N(v)],x) 
$$

has a real root $-1 \leq R <0$. 
Now considering formula (\ref{DerivatCliq1}), we immediately conclude that 
$\frac{d}{dx}C(G,x)$ has a real root. Thus the graphs $G$ has only clique roots. 
\end{proof} 
 
\begin{prop}
Let $G$ be a $K_{4}$-free connected chordal graph. 
Then $G$ has only clique roots. 	
\end{prop}

\begin{proof}[Sketch of Proof]
Since the neighborhood of each vertex $v$ is a \emph{forest} $F_{v}$, then considering the fact that the smallest clique root $r_{F_v}$  of $F_{v}$ satisfies 
$r_{F_{v}} \leq -1$ from Corollary \ref{keyForest1} we conclude that $r_{F} \leq R_{F}$. Thus the key lemma implies that 
$
\frac{d}{dx}C(G,x)
$   
has a real root. Now considering the fact $C(G,x)$ has a clique root $-1$ and $
\frac{d}{dx}C(G,x)
$ has a real root and the proof is complete.  
\end{proof} 

\begin{prop}\label{keyresult1}
Let $G$ be a bi-connected $K_{5}$-free chordal graph. Then $G$ has only clique roots. 	
\end{prop}

\begin{proof}[Sketch of Proof]
considering the fact that $C(G,x)$ is a quartic polynomial with at least two clique roots $r=-1$, we just need to prove $\frac{d^2}{dx^2} C(G,x)$ has a real root. 
Now, since for each $e\in E(G)$ the graph 
$G[N(e)]$ is triangle-free, based on combinatorial formula (\ref{DerivatCliq2}) and Lemma \ref{keylem1} the proof is complete. 	
\end{proof}

\section{Open Questions and Conjectures}

Considering Lemma \ref{keylem1}, we come up with the following conjecture. 

\begin{conj}
Let $G$ be a connected $K_{4}$-free graphs, then 
$G$ has only clique roots. 
\end{conj}

Based on Proposition \ref{keyresult1}, the following open question naturally arises. 

\begin{quest}
Is there any $2$-connected non-chordal $K_{5}$-free graph which
has only clique roots?   
\end{quest}
Consider the wheel graph $W_{5}$ on five vertices with on extra 
edge connecting two non-adjacent vertices on the outer ring. Now it is clear that 
$
C(G,x) = (1+x)(1+5x+6x^2+x^{3})
$
has only real roots. 
\begin{conj}
The class of all $2$
-connected 
$K_{5}$-free graphs in which each edge belongs to at most two triangles has only clique roots. 
\end{conj} 
\begin{quest}
Which classes of $K_{r}$-free chordal graphs has only real roots.	
\end{quest}
We finally came up with following stronger conjecture. 
\begin{conj}
The class of $l$-connected chordal graphs 
which are $K_{l+3}$-free has only clique roots. 
\end{conj}
We believe that the last conjecture is a starting point to find another algebraic graph-theoretic proof of the well-known Tur\'{a}n-type graph theorems. 

\end{document}